\documentclass[10pt]{amsart}
\usepackage{amssymb}
\usepackage{epsfig}
\usepackage{url}
\usepackage{setspace}
\usepackage{pdflscape}
\theoremstyle{plain}

\usepackage{enumerate} 
\usepackage{tikz} 
\usetikzlibrary{arrows}
\usepackage{caption} 

\newtheorem*{thm*}{Theorem}
\newtheorem{thm}{Theorem}[section]

\newtheorem{lem}[thm]{Lemma}

\renewcommand{\phi}{\varphi}

\DeclareMathOperator{\rad}{rad}

\begin{document}

\title[On decompositions of quadrinomials and related Diophantine equations]{On decompositions of quadrinomials and related Diophantine equations}
\author{Maciej Gawron}

\footnote{The research of the author was supported by Polish National Science Centre grant UMO-
2014/13/N/ST1/02471}

\begin{abstract}
Let $A,B,C,D$ be rational numbers such that $ABC \neq 0$, and let $n_1>n_2>n_3>0$ be positive integers. 
We solve the equation
\begin{equation*}
 Ax^{n_1}+Bx^{n_2}+Cx^{n_3}+D = f(g(x)),
\end{equation*}
in $f,g \in \mathbb{Q}[x]$. In sequel we use Bilu-Tichy method to prove finitness of integral solutions of the equations
\begin{equation*}
Ax^{n_1}+Bx^{n_2}+Cx^{n_3}+D = Ey^{m_1}+Fy^{m_2}+Gy^{m_3}+H,
\end{equation*}
where $A,B,C,D,E,F,G,H$ are rational numbers $ABCEFG \neq 0$ and $n_1>n_2>n_3>0$, $m_1>m_2>m_3>0$, $\gcd(n_1,n_2,n_3) = \gcd(m_1,m_2,m_3)=1$
and $n_1,m_1 \geq 9$. And the equation
\begin{equation*} 
A_1x^{n_1}+A_2x^{n_2}+\ldots+A_l x^{n_l} + A_{l+1} = Ey^{m_1}+Fy^{m_2}+Gy^{m_3},
\end{equation*}
where $l \geq 4$ is fixed integer, $A_1,\ldots,A_{l+1},E,F,G$ are non-zero rational numbers, except for possibly $A_{l+1}$,
$n_1>n_2>\ldots > n_l>0$, $m_1>m_2>m_3>0$ are positive integers such that $\gcd(n_1,n_2, \ldots n_l) = \gcd(m_1,m_2,m_3)=1$, and $n_1 \geq 4$, $m_1 \geq 2l(l-1)$.

\end{abstract}

\maketitle

\section{Introduction}
In paper \cite{SCH1} Schinzel, Pint\'{e}r and  P\'{e}ter give an inefficient criterion for the Diophantine equation of the form
\begin{equation*}
 ax^m+bx^n+c = dy^p+e^q,
\end{equation*}
where $a,b,c,d,e$ rationals, $ab\neq 0 \neq de$, $m>n>0, p>q>0$, $\gcd(m,n)=1$, $\gcd(p,q)=1$, and $m,p \geq 3$ to have
infinitely many integer solutions. 

In the later paper Schinzel \cite{SCH2} dropped the assumption $\gcd(m,n)=1$, $\gcd(p,q)=1$ and gives a necessary and sufficient condition
for such equation to have infinitely many integer solutions.

In the recent paper Kreso \cite{KRSNY}
proved the finitness of integral solutions for the equation
\begin{equation*}
 a_1x^{n_1}+a_2x^{n_2}+\ldots+a_lx^{n_l}+a_{l+1} =b_1y^{m_1} + b_2y^{m_2},
\end{equation*}
where $l \geq 2$ and  $m_1 > m_2$, $n_1>n_2>\ldots > n_l$ are fixed positive integers satisfying \mbox{$\gcd(m_1,m_2)=1$,} $\gcd(n_1,n_2,\ldots,n_l) = 1$,
$a_1,a_2,\ldots,a_l,a_{l+1},b_1,b_2$ are non-zero rationals, except for possibly $a_{l+1}$. With $n_1 \geq 3, m_1 \geq 2l(l-1)$ and $(n_1,n_2) \neq (m_1,m_2)$.

All the mentioned results relies on Bilu-Tichy Theorem \cite{BTCH}, and theorems concerning decompositions of trinomials \cite{SF} as main ingredients.
No such results for the equations involving at least three non-zero coefficients at positive powers
on both sides are known 
mainly because we have no results concerning decompositions of lacunary polynomials with more than three non-zero terms~\cite{KRSNY}.
Some partial results in this direction are given in \cite{KRTCH}.

In this note we describe all possible decompositions of quadrinomials.
In the sequel we use Bilu-Tichy theorem to prove the following generalizations of Schinzel and Kreso results. 
More precisely, we prove the following
\begin{thm*}[A] Let $f(x) = Ax^{n_1}+Bx^{n_2}+Cx^{n_3}+D$, $g(x) = Ex^{m_1}+Fx^{m_2}+Gx^{m_3}+H$ with $f,g \in \mathbb{Q}[x]$, $n_1>n_2>n_3>0$, $m_1>m_2>m_3>0$, and 
 $\gcd(n_1,n_2,n_3) = 1$, $\gcd(m_1,m_2,m_3)=1$, $(m_1,m_2,m_3) \neq (n_1,n_2,n_3)$, $ABC \neq 0$, $EFG \neq 0$ and $n_1,m_1 \geq 9$. Then the equation
 \begin{equation*}
  f(x) = g(y)
 \end{equation*}
has only finitely many integer solutions.
\end{thm*}

\begin{thm*}[B]
 Let $l \geq 4$ and $n_1>n_2> \ldots > n_l>0$, $m_1>m_2>m_3>0$ be positive integers. Let
 \begin{equation*}
 f(x) = A_1x^{n_1}+A_2x^{n_2}+\ldots+A_l x^{n_l} + A_{l+1} \quad \text{ and } \quad  g(x) = Ex^{m_1}+Fx^{m_2}+Gx^{m_3}
 \end{equation*}
 be polynomials with rational coefficients such that $\gcd(n_1,n_2,\ldots,n_l) = 1$, $\gcd(m_1,m_2,m_3)=1$,  
 $A_1A_2\ldots A_l \neq 0$, $EFG \neq 0$ and $m_1 \geq 2l(l-1), n_1 \geq 4$. Then the equation 
 \begin{equation*}
  f(x) = g(y)
 \end{equation*}
has only finitely many integer solutions.
 \end{thm*}

Our results are ineffective as we use Theorem of Bilu and Tichy which relies on classical theorem of Siegel on integral points.

\section{Decompositions of quadrinomials}
In this section we describe decompositions of quadrinomials. We will use some classical lemmas.
Let us recall Mason-Stothers Theorem \cite{MS1,MS2}.
\begin{thm}\label{MST} 
Let $a(t),b(t),c(t) \in K[t]$ be relatively prime polynomials over a field of characteristic zero, such that $a+b=c$, and not all of them are constant.
Then 
\begin{equation*}
 \max \{ \deg a, \deg b, \deg c \} \leq \deg (\rad(abc))-1,
\end{equation*}
where $\rad(f)$ is the product of the distinct irreducible factors of $f$. 

\end{thm}

Let us recall Haj\'{o}s lemma \cite{HAJO}.

\begin{lem}\label{HAJ} Let $K$ be a field o characteristic $0$. 
If $f(x) \in K[X]$ has a root $z \neq 0$ of multiplicity $n$ then $f$ has at least $n+1$ terms.
\end{lem}
\begin{proof}
 We use induction on $n$. When $n=1$ then the statement obviously holds. For $n > 1$ let us put $f(x) = x^kf_1(x)$ where $f_1(0) \neq 0$.
 Then $z$ is a root of $f_1'$ of multiplicity $n-1$ and $f_1'$ has exactly one term less than $f$. The result follows.
\end{proof}

\begin{lem}\label{pierwiastek_z_trojmianu}
 Let $K$ be a field o characteristic $0$. If $f(x) \in K[x]$ satisfy the equation $f(x)^2 = x^{n_1}+Ax^{n_2}+B$ for some $A,B \in K \setminus \{0\}$ and
 $n_1 > n_2 >0$ then $f(x)$ is a binomial.
\end{lem}
\begin{proof}
 Suppose otherwise, that $f$ has at least three non-zero terms. Let us write
 \begin{equation*}
  f(x) = x^{k_1}+ Ux^{k_2}+\ldots+Vx^{k_3}+W,
 \end{equation*}
 for some $k_1>k_2 \geq k_3 > 0$, and $UVW\neq 0$. Then we have
 \begin{equation*}
  f(x)^2 = x^{2k_1}+ 2Ux^{k_1+k_2}+\ldots+2VWx^{k_3}+W^2,
 \end{equation*}
 so $f(x)^2$ has at least four non-zero terms.
\end{proof}

\begin{lem}\label{lem3} Let $K$ be an algebraically closed field of characteristic $0$. Let $f,g,h \in K[x]$ be polynomials such that
$f(x) = g(h(x))$ and $ \deg g > 1$ then there exists $\gamma \in K$ such that 
\begin{equation*}
\deg (\gcd(f(x)-\gamma, f'(x)) \geq \deg h. 
\end{equation*}
\end{lem}
\begin{proof}
 Let $\beta$ be a root of $g'(x)$. We define $\gamma = g(\beta)$ then $h(x) - \beta$ divides both $f'(x)$ and $f(x)-\gamma$.
\end{proof}

Now we are ready to state the main theorem of this section.

\begin{thm}\label{DECOMP}
 Let $K$ be an algebraically closed field of characteristic $0$. Let $f(x) = Ax^{n_1}+Bx^{n_2}+Cx^{n_3}+D$ for some $A,B,C,D \in K$ such that $ABC \neq 0$ 
 and $n_1 > n_2 > n_3 > 0$. 
 Suppose that $f(x) = g(h(x))$ for some $g,h \in K[x]$ then one of the following cases holds
 \begin{enumerate}
  \item $g(x) = (Ax^{\frac{n_1}{d}}+Bx^{\frac{n_2}{d}}+Cx^{\frac{n_3}{d}}+D) \circ l^{-1}$ and $h(x) = l \circ x^d $ for some linear polynomial $l \in K[x]$, and
  positive integer $d | \gcd(n_1,n_2,n_3)$,
  \item $g(x) = l(x)$ and $h(x) = l^{-1} \circ f(x)$ for some linear $l \in K[x]$   
  \item $g(x) = (Ax^2+D) \circ l$ and $h(x) = l^{-1} \circ (x^{ \frac{n_1}{2}} + \frac{B}{2A} x^{\frac{n_3}{2}})$ where $l \in K[x]$ 
  is some linear polynomial. Moreover the following conditions holds $2n_2 = n_1+n_3$ and $C = \frac{B^2}{4A}$.
  \item $g(x) = (Ax(x-c^2)+D) \circ l $ and $h(x) = l^{-1} \circ (x^{2n_3}+cx^{n_3})$ for some linear $l \in K[x]$, and non-zero $c$. 
  Moreover the following conditions holds
  $n_1=4n_3,n_2=3n_3$
 \end{enumerate}
\end{thm}

\begin{proof}
By replacing $f$ and $g$ by $(Ax+D)^{-1} \circ f$ and $(Ax+D)^{-1} \circ g$ we can assume that $A=1, D=0$. Moreover by replacing
$g,h$ by $g \circ l^{-1}$ and $l \circ h$ for suitable linear $l$, we can assume that $g,h$ are monic, and $g(0)=h(0)=0$.

Let us write
\begin{equation*}
 g(x) = x^{a_0} (x-x_1)^{a_1} \cdots (x-x_k)^{a_k},
\end{equation*}
with $a_0,a_1,\ldots,a_k \in \mathbb{N}_+$, $x_i \neq 0$ for $i=1,2,\ldots, k$ and $x_i \neq x_j$ for $i \neq j$, $i,j=1,2,\ldots,n$.
We have 
\begin{equation*}
 h(x)^{a_0}(h(x)-x_1)^{a_1} \cdots (h(x)-x_k)^{a_k} = x^{n_1}+Bx^{n_2} + Cx^{n_3}.
\end{equation*}
We write $h(x) = x^d h_1(x)$, where $h_1(x)$ is some monic polynomial such that $h_1(0) \neq 0$. 
If $h_1 \equiv 1$ then we get $h(x) = x^d$ and $g(x) = x^{n_1/d} + Bx^{n_2/d} + Cx^{n_3/d}$. This corresponds to the first case  on our list.

Let us suppose that  
$h_1$ has a non-zero root $\xi$. Then from Lemma \ref{HAJ} we get that the multiplicity of $\xi$ as a root of $x^{n_1}+Bx^{n_2}+Cx^{n_3}$  
is less or equal than $2$, and therefore
$a_0 \in \{1,2\}$. We consider these two cases separately.

\textbf{Case 1: $a_0=1.$}\\
If $g(x) = x$ then we get trivial decomposition i.e. the second case on our list. Suppose that $\deg g \geq 2$. We have that
\begin{equation*}
 x^{n_1}+Bx^{n_2}+Cx^{n_3} = g(x^{n_3}h_1(x)).
\end{equation*}
Let us prove that $h_1(x) = h_2(x^{n_3})$ for some polynomial $h_2(x)$. Suppose otherwise, 
let $h_1(x)$ has non-zero coefficient $c_{\nu}$ at power $x^{\nu}$, $n_3 \nmid \nu$ and
$\nu$ is the smallest integer with this properties. Let us prove that $g(h(x))$ has at least four non-zero coefficients. We have
\begin{equation*}
 g(h(x)) = x^{n_3}h_1(x)(x^{n_3} h_1(x)-x_1)^{a_1} \cdots (x^{n_3}h_1(x)-x_k)^{a_k} = \sum_{i=1}^{\deg g} C_i (x^{n_3}h_1(x))^i,
\end{equation*}
where $g(x) = \displaystyle\sum_{i=1}^{\deg g} C_i x^i$. Polynomial $h_1(x)$ is not a monomial and thus $C_{\deg g}(x^{n_3}h_1(x))^{\deg g}$ 
has at least two-non zero coefficients at powers which cannot cancel with $C_{j}(x^{n_3}h_1(x))^{j}$ where $j<\deg g$. 
The coefficient at $x^{n_3 + \nu}$ in $g(h(x))$ is equal to $C_1c_v\neq 0$ it can't cancel as a coefficient at the lowest power which
is not divisible by $n_3$. So in that case $g(h(x))$ has at least four non-zero coefficients - a contradiction. We proved that $h_1(x) = h_2(x^{n_3})$.
As a consequence we get the equality 

\begin{equation*}
 x^{n_1}+Bx^{n_2}+Cx^{n_3} = g(x^{n_3}h_2(x^{n_3})).
\end{equation*}
Therefore $n_3|n_1, n_3|n_2$ say $m_1n_3=n_1, m_2n_3=n_2$ and 
\begin{equation}{\label{eq4}}
x^{m_1}+Bx^{m_2}+Cx = g(xh_2(x)).
 \end{equation}
Let us write $k=m_1-m_2$. We claim that $h_2(x) = h_3(x^k)$ for some $h_3 \in K[x]$. 
By comparison of coefficients in identity (\ref{eq4}) we get that 
\begin{equation*}
h_2(x) = x^t+\frac{B}{\deg g} x^{t-k} + \text{l.o.t.},
\end{equation*}
for some $t$. Let us prove that if a coefficient $D_{s}$ at $x^{s}$ in 
$h_2(x)$ is non-zero then $s \equiv t \pmod k$. Suppose otherwise, let $\nu$ be the highest power at which $h_2(x)$ has coefficient
$D_{\nu}$ which is non-zero and $\nu \not\equiv t \pmod k$. We have
\begin{equation*}
 x^{m_1}+Bx^{m_2}+Cx = x^{\deg g} \left(x^{t}+\frac{B}{\deg g}x^{t-k} + \text{l.o.t.}\right)^{\deg g} + \text{l.o.t.}.
\end{equation*}
Let us observe that the coefficient at $x^{\deg g+(\deg g-1)t+\nu}$ on the right hand side is equal to $D_{\nu}\deg g $. It cannot cancel
because it is the coefficient at the highest power $x^u$ which satisfies $u \neq \deg g + t \deg g \pmod k$.
Of course $m_2 = \deg g+(\deg g-1)t + t-k > \deg g+(\deg g-1)t+\nu > 1$ so we arrive at contradiction. We know that $h_2(0) \neq 0$ so we have
$h_2(x) = h_3(x^k)$. 

We thus proved that
\begin{equation*}
 g(xh_2(x)) = g(xh_3(x^k)) = x^{m_1}+Bx^{m_2}+C.
\end{equation*}
We prove that $g(x) = xg_2(x^k)$ for some polynomial $g_2(x)$. Suppose that this is not the case and put 
$g(x) = \displaystyle\sum_{i=1}^{\deg g} C_i x^i$. Let $\nu$ be the smallest integer such that $ \nu \not \equiv 1 \pmod k$ and $C_{\nu} \neq 0$. We have
\begin{equation*}
 \displaystyle\sum_{i=1}^{\deg g} C_i (xh_3(x^k))^i = x^{m_1} + Bx^{m_2}+C.
\end{equation*}
The coefficient at $x^{\nu}$ on the left hand side is equal to $C_\nu h_3(0) \neq 0$. Thus $\nu = m_1$ or $\nu = m_2$. On the other hand
\begin{equation*}
 m_2=m_1-k = \deg(g(xh_3(x))) -k \geq \nu (k+1) -k = (\nu-1)k + \nu > \nu, 
\end{equation*}
so we arrive at contradiction.

We have
\begin{equation*}
 xh_3(x^k)g_2(x^kh_3(x^k)^k) = x^{m_1}+Bx^{m_2}+Cx,
\end{equation*}
therefore $m_2-1 = kl$ and $m_1-1=k(l+1)$ for some $l$. In consequence we get
\begin{equation*}
 h_3(x)g_2(xh_3(x)^k) = x^{l+1}+Bx^{l}+C.
\end{equation*}
Let us prove that $k=1$. We write $g_2(x) = \sum_{j=0}^{\deg g_2} W_j x^j$, and get
\begin{equation*}
 (W_0h_3(x)-C) + (xh_3(x)^{k+1})\sum_{j=1}^{\deg g_2} W_j(xh_3(x)^k)^{j-1} = x^{l+1}+Bx^{l}.
\end{equation*}
We apply Theorem \ref{MST} to the above equation and get
\begin{equation*}
l+1 \leq \deg h_3 + (l+1 -  k\deg h_3) + 1 - 1 = l+1 -(k-1)\deg h_3, 
\end{equation*}
thus $(k-1) \deg h_3 \leq 0$ and $k=1$. We get that 
\begin{equation*}
 g(xh_2(x)) = xh_2(x)g_2(xh_2(x)) = x^{l+2}+Bx^{l+1}+Cx.
\end{equation*}
Let us write $F(x) = g(xh_2(x))$. From Lemma \ref{lem3} we get that there exists $\lambda \in K$ such that 
\begin{equation*}
\deg(\gcd(F(x)-\lambda, F'(x)) \geq \deg xh_2(x).
\end{equation*}
We apply Theorem \ref{MST} to the equation
\begin{equation*}
 F(x) - \lambda = (x^{l+2}+Bx^{l+1})+(Cx-\lambda),  
\end{equation*}
and get
\begin{equation*}
 l+2 \leq 2+1+ (\deg (F(x) -\lambda) - \deg(\gcd(F(x)-\lambda, F'(x))) -1 \leq l+4 - \deg (xh_2(x)), 
\end{equation*}
therefore $\deg h_2(x) \leq 1$. Let $h_2(x) = x+c$ for some non-zero $c$, then we have
\begin{equation*}
 g(x(x+c)) = x^{l+2}+Bx^{l+1}+Cx.
\end{equation*}
The left hand side is symmetric with respect to the line $x = - \frac{c}{2}$, and thus so is right hand side. We get
\begin{equation*}
 x^{l+2}+Bx^{l+1}+Cx = (-x-c)^{l+2}+B(-x-c)^{l+1}+C(-x-c).
\end{equation*}
We compute second derivative of both sides and get
\begin{equation*}
 (l+2)(l+1)x^{l}+B(l+1)lx^{l-1} = (l+2)(l+1)(-x-c)^{l}+B(l+1)l(-x-c)^{l-1}.
\end{equation*}
which is equivalent to
\begin{equation*}
 ((l+2)x+Bl)x^{l-1} = (x+c)^{l-1}((l+2)(x+c)-Bl),
\end{equation*}
as $c\neq 0$ we have $l=2$. As a consequence we get $\deg g = 2$, and $g(x) = x(x+b)$ for some non-zero $b$. Finally we have that the coefficient at $x^2$
in $g(x(x+c))$ is equal to zero which implies $b=-c^2$. Summing up: in case of $a_0=1$ we get the solution of $g(h(x))=f(x)$ of the following form 
\begin{equation*}
 g(x) = x^2-c^2x, \quad h(x) = x^{n_3}(x^{n_3}+c)
\end{equation*}
which corresponds to the third case on our list.

\textbf{Case 2: $a_0=2$.}

We know that $a_0d = n_3$ therefore $d = \frac{n_3}{2}$. If $g(x) = x^2$ then $h(x)^2 = x^{n_1}+Bx^{n_2} + Cx^{n_3}$,
and $h_1(x)^2 = x^{n_1-n_3} + Bx^{n_2-n_3}+C$. From Lemma \ref{pierwiastek_z_trojmianu} we get that $h(x)$ is a binomial, which corresponds to the second case on our
list.

Now let $g(x) \neq x^2$, so $g$ has at least one non-zero root, and $\deg g(x) \geq 3$. 
Let us prove that $h_1(x) = h_2(x^{\frac{n_3}{2}})$, for some monic polynomial $h_2$. Suppose that this is not the case and assume that 
$h_1(x)$ has a non-zero coefficient $c_{\nu}$ at power $x^{\nu}$, $d = \frac{n_3}{2}  \nmid \nu$. Choose 
$\nu$ as the smallest integer with this property. We prove that $g(h(x))$ has at least four non-zero coefficients. We have
\begin{equation*}
 g(h(x)) = x^{n_3}h_1(x)^{2}(x^d h_1(x)-x_1)^{a_1} \cdots (x^dh_1(x)-x_k)^{a_k} = \sum_{i=2}^{\deg g} C_i (x^dh_1(x))^i,
\end{equation*}
where $g(x) = \displaystyle\sum_{i=2}^{\deg g} C_i x^i$. Polynomial $h_1(x)$ is not a monomial and thus $C_{\deg g}(x^dh_1(x))^{\deg g}$ 
has at least two-non zero coefficients, at powers which cannot cancel with $C_{j}(x^dh_1(x))^{j}$ where $j<\deg g$. 
Moreover, the coefficient at $x^{n_3 + \nu}$ in $g(h(x))$ is equal to $2C_2c_vh_1(0) \neq 0$. It cannot cancel as a coefficient at the lowest power which
is not divisible by $d$. So in that case $g(h(x))$ has at least four non-zero coefficients - a contradiction. We proved that $h_1(x) = h_2(x^{n_3/2})$
and thus
\begin{equation*}
 x^{n_1}+Bx^{n_2}+Cx^{n_3} = g(x^{d}h_2(x^d)),
\end{equation*}
where $d=n_3/2$. Therefore $d|n_1, d|n_2$ say $m_1d=n_1, m_2d=n_2$ and 
\begin{equation}{\label{eq1}}
x^{m_1}+Bx^{m_2}+Cx^2 = g(xh_2(x)).
 \end{equation}
Let us write $k=m_1-m_2$. We claim that $h_2(x) = h_3(x^k)$. By comparison of coefficients in the equality (\ref{eq1}) we get that
\begin{equation*}
h_2(x) = x^t+\frac{B}{\deg g} x^{t-k} + \text{l.o.t.},
\end{equation*}

for some $t$. Let us prove that if a coefficient $D_{s}$ at $x^{s}$ in 
$h_2(x)$ is non-zero then $s \equiv t \pmod k$. Suppose otherwise, let $\nu$ be the highest power at which $h_2(x)$ has coefficient
$D_{\nu}$ which is non-zero and $\nu \not\equiv t \pmod k$. We have
\begin{equation*}
 x^{m_1}+Bx^{m_2}+Cx^2 = x^{\deg g} \left(x^{t}+\frac{B}{\deg g}x^{t-k} + \text{l.o.t.} \right)^{\deg g} + \text{l.o.t.}.
\end{equation*}
Let us observe that the coefficient at $x^{\deg g+(\deg g-1)t+\nu}$ on the right hand side is equal to $D_{\nu}\deg g $. It cannot cancel
because it is the coefficient at the highest power $x^u$ which satisfies $u \neq \deg g + t \deg g \pmod k$.
Of course $m_2 = \deg g+(\deg g-1)t + t-k > \deg g+(\deg g-1)t+\nu > 2$ so we arrive at contradiction. We know that $h_2(0) \neq 0$ and thus
$h_2(x) = h_3(x^k)$. 

We can write
\begin{equation*}
 x^2h_3(x^k)^2| g(xh_2(x)) = x^{m_1}+Bx^{m_2}+Cx^2, 
\end{equation*}
therefore
\begin{align*}
 h_3(x^k) &| \frac{d}{dx} (x^{m_1-2}+Bx^{m_2-2}+C) \\ &= (m_1-2)x^{m_1-3}+(m_2-2)Bx^{m_2-3} \\ &= x^{m_2-3}((m_1-2)x^k+(m_2-2)B).  
\end{align*}
We know that $h_3(0) \neq 0$ so $h_3(x^k) | (m_1-2)x^k+(m_2-2)B$, in particular $ \deg h_3 \leq 1$. If $h_3(x)$ is constant then so is $h(x)$
and we arrive at contradiction. 

Suppose that $h_3(x)$ is linear then $h_2(x) = x^k+c$ for some non-zero $c$. Let us write $g(x) = x^2g_2(x)$. We then have
\begin{equation*}
 x^{m_1}+Bx^{m_2}+Cx^2 = x^2(x^k+c)^2 g_2(x(x^k+c)),
\end{equation*}
and thus
\begin{equation*} 
 x^{m_1-2}+Bx^{m_2-2}+C = (x^k+c)^2 g_2(x(x^k+c)),
\end{equation*}
Let us prove that $g_2(x) = g_3(x^k)$ for some $g_3(x)$. Suppose that this is not the case.  
Let $x^u$ be the lowest power such that $g_2$ has non-zero coefficient at 
$x^u$ and $ u \not \equiv 0 \pmod k$ then we have that the coefficient at $x^u$ in $(x^k+c)^2 g_2(x(x^k+c))$ is non-zero therefore $u =m_2-2$ or $u=m_1-2$
in both cases we have
\begin{equation*}
 m_2-2 = (m_1-2)-k = (2k+ \deg g_2 \cdot (k+1))-k \geq k + u (k+1) \geq k+u \geq k+(m_2-2) > m_2-2.  
\end{equation*}
So we have
\begin{equation*}
 x^{m_1-2}+Bx^{m_2-2}+C = (x^k+c)^2g_3(x^k(x^k+c)^k),
\end{equation*}
in particular $k|m_1-2$ and $k|m_2-2$, say $k(s+1) = m_1-2$. We have
\begin{equation}\label{eq2}
 x^{s+1}+Bx^{s} + C = (x+c)^2g_3(x(x+c)^k).
\end{equation}
Let us put $g_3(x) = \sum_{j=0}^{\deg g_3} W_j x^j$, and write
\begin{equation*}
 x^{s+1}+Bx^{s} = (W_0(x+c)^2-C) + x(x+c)^{k+2} \left(\sum_{j=1}^{\deg g_3} W_j (x(x+c)^k)^{j-1}\right).
\end{equation*}
We apply Theorem \ref{MST} to the above equation and get
\begin{equation*}
 s+1 \leq (2 + (s+1-(k+2)) + 2) -1, 
\end{equation*}
so $k=1$. We plug this information into equation (\ref{eq2}) and get
\begin{equation}\label{eq3}
  x^{s+1}+Bx^{s}+C = (x+c)^2 g_2(x(x+c)).
\end{equation}
In consequence
\begin{equation*}
 x^{s+3}+Bx^{s+2}+Cx^2 = (x+c)^2x^2 g_2(x(x+c)), 
\end{equation*}
the left hand side is symmetric with respect to the line $x=-\frac{c}{2}$, and thus so is right hand side
\begin{equation*}
 x^{s+3}+Bx^{s+2}+Cx^2 = (-x-c)^{s+3}+B(-x-c)^{s+2}+C(-x-c)^2.
\end{equation*}
We compute second derivative of both sides and get 
\begin{equation*}
 (s+3)(s+2)x^{s+1}+(s+2)(s+1)Bx^{s}= (s+3)(s+2)(-x-c)^{s+1}+(s+2)(s+1)B(-x-c)^{s}.
\end{equation*}
If $s>1$ then the only multiple root of left hand side is $x=0$, and only multiple root of the right hand side is $x=-c$, thus $s=1$.
By comparing degrees in equation (\ref{eq3}) we get that $ \deg g_2 = 0$, contradiction as $g_2(x)$ has non-zero root.

\end{proof}

\section{A Diophantine equation}
In this section we give sufficient condition on quadrinomials to have only finitely many common integral points.
We recall Bilu-Tichy result. To state this theorem we will need the notion of standard pairs over $\mathbb{Q}$.
We list standard pairs of polynomials over $\mathbb{Q}$ in the table
\vskip 0.2 in

\begin{center}

\begin{tabular}{|c|c|c|}
\hline
kind & standard pair (or switched) & parameter restrictions \\
\hline
first & $(x^m, ax^rp(x)^m)$ & $r<m$, $\gcd(r,m)=1$, $r+ \deg p > 0$ \\
\hline
second & $(x^2,(ax^2+b)p(x)^2)$ & \\
\hline
third & $(D_m(x,a^n),D_n(x,a^m))$ & $\gcd(m,n) = 1$ \\
\hline
fourth & $(a^{-m/2}D_m(x,a),-b^{-n/2}D_n(x,b))$ & $\gcd(m,n)=2$ \\
\hline
fifth & $((ax^2-1)^3, 3x^4-4x^3)$ & \\
\hline
\end{tabular}

\end{center}

\vskip 0.2 in

where $a,b$ are non-zero rationals, $m,n$ are positive integers, $r$ is non-negative integer, $p \in \mathbb{Q}[x]$ is a polynomial (which may be
constant) and $D_n(x,a)$ is the $n$-th Dickson polynomial with parameter $a$ given by the formula
\begin{equation*}
 D_n(x,a) = \sum_{i=0}^{[n/2]} \frac{n}{n-i} \binom{n-i}{i}(-a)^i x^{n-2i}.
\end{equation*}

Now we are ready to recall the theorem of Bilu and Tichy \cite{BTCH}

\begin{thm}\label{BTT}
 Let $f,g \in \mathbb{Q}[x]$ be non-constant polynomials. Then the following assertions are equivalent
 \begin{itemize}
  \item The equation $f(x) = g(y)$ has infinitely many rational solutions with a bounded denominator.
  \item We have 
  \begin{equation*}
  f(x) = \varphi (f_1(\lambda(x))), \quad g(x) = \varphi(g_1(\mu(x))), 
  \end{equation*}
where $\varphi \in \mathbb{Q}[x]$, $\lambda, \mu \in \mathbb{Q}[x]$ are linear polynomials, and $(f_1,g_1)$ is a standard
pair over $\mathbb{Q}$ such that the equation $f_1(x) = g_1(y)$ has infinitely many rational solutions with a bounded denominator.
 \end{itemize}

\end{thm}

Before stating the main theorem we prove some lemmas concerning decompositions of certain polynomials. As a main ingredient
in proofs we will use the classical theorem of Gessel and Viennot~\cite{GV} concerning determinants with binomial coefficients. 
\begin{thm}\label{GVL}
 Let $0\leq a_1 < a_2 < \ldots < a_n$ and $0 \leq b_1 < b_2 < \ldots < b_n$ be strictly increasing sequences of non-negative integers. Then the determinant
 \begin{equation*}
  \det \left( \left[\binom{a_i}{b_j}\right]_{i,j=1,2,\ldots,n} \right) 
 \end{equation*}
is non-negative, and positive iff $b_i \leq a_i$ for all $i$. 
\end{thm}

\begin{lem}\label{DZIURY}
 Let $f,g \in \mathbb{Q}[x]$ be polynomials with rational coefficients, and $u,v \in \mathbb{Q}$ be non-zero rationals such that
\begin{equation*}
 f(x) = g(ux+v).
\end{equation*}
Suppose that $g(x)$ has exactly $l$ non-zero terms and $f(x)$ has exactly $k$ non-zero terms, and $n = \deg f = \deg g$. Then the following
inequality holds
\begin{equation*}
n+2 \leq k+l. 
\end{equation*}
\end{lem}
\begin{proof}
We put $g(x) = \displaystyle \sum_{i=1}^l C_{n_i} x^{n_i}$, where $(n_i)_{i=1}^l$ is a decreasing sequence of non-zero integers, and $C_{n_i}$ are non-zero rationals. The coefficient at $x^j$ in the polynomial $g(ux+v)$ is equal to 
\begin{equation*}
 u^jv^{-j} \sum_{i=1}^l C_{n_i} \binom{n_i}{j} v^{n_i}. 
\end{equation*}
Let us suppose that this coefficient vanish for $j=m_1,m_2,\ldots,m_s$, where $s = n+1-k \geq l$ and $(m_i)_{i=1}^s$ is a decreasing sequence of non-zero integers.  We observe that the vector $(C_{n_1}v^{n_1},\ldots,C_{n_l}v^{n_l})$ is perpendicular to every row
of the matrix
\begin{equation*}
 \left[ \binom{n_i}{m_j} \right]_{i=1,\ldots,l, j = 1,\ldots, s}.
\end{equation*}
We put $t = \max \{  i | n_i \geq m_i,  1\leq i \leq l \} $, of course $n = n_1 > m_1$ and so $t$ is well-defined. From Lemma \ref{GVL} we get that the determinant of the matrix 
\begin{equation*}
 M = \left[ \binom{n_i}{m_j} \right]_{i=1,\ldots,t, j = 1,\ldots, t}
\end{equation*}
is non-zero. Moreover the vector $(C_{n_1}v^{n_1},\ldots,C_{n_t}v^{n_t})$ is in the kernel of $M$, therefore $v=0$. Which
contradicts the assumption $v \neq 0$. We proved that $s = n+1-k < l$, therefore $n+2 \leq k+l$.  
\end{proof}

\begin{lem}\label{DICKSON}
Let $n_1>n_2>\ldots >n_s>0$ be positive integers and $A_1,A_2,\ldots,A_{s+1}$ be rational numbers such that $A_1A_2\ldots A_s \neq 0$, 
and put $f(x) = A_1x^{n_1}+\ldots+A_sx^{n_s}+A_{s+1}$. If the equation   
 \begin{equation*}
 D_{n_1}(x, \gamma) = f(ux+v)
 \end{equation*}
 holds for some $u,v, \gamma \in \mathbb{Q}$ such that $u\gamma \neq 0$,
 then $n_1 \leq 2s$. 
\end{lem}
\begin{proof}
 If $v=0$ then we get 
 \begin{equation*}
 D_{n_1}(x, \gamma) = \sum_{i=0}^{[n_1/2]} \frac{n_1}{n_1-i} \binom{n_1-i}{i}(-\gamma)^i x^{n_1-2i} = f(ux).
 \end{equation*}
The polynomial on the left hand side has exactly $[n_1/2]+1$ non-zero coefficients, while $f(ux)$ has $s$ or $s+1$ non-zero coefficients. In consequence
$[n_1/2] +1 \leq s+1$ so $n_1 \leq 2s+1$. It can't be $n_1=2s+1$ because in that case $A_{s+1} = 0$ and $f(ux)$ has $s$ non-zero coefficients
so $[n_1/2]+1 \leq s$ and thus $n_1 \leq 2s-1$.

Suppose that $ v \neq 0$ and $n_1 \geq 2s+1$. From Lemma \ref{DZIURY} we get that 
\begin{equation*}
\deg D_{n_1}(x,\gamma) = n_1 \leq [n_1/2]+1+(s+1)-2 = [n_1/2]+s, 
\end{equation*}
therefore $n_1 \leq 2s$.

\end{proof}

Now we are ready to state the main theorems of this section.

\begin{thm*}[A] Let $f(x) = Ax^{n_1}+Bx^{n_2}+Cx^{n_3}+D$, $g(x) = Ex^{m_1}+Fx^{m_2}+Gx^{m_3}+H$ with $f,g \in \mathbb{Q}[x]$, $n_1>n_2>n_3>0$, $m_1>m_2>m_3>0$, and 
 $\gcd(n_1,n_2,n_3) = 1$, $\gcd(m_1,m_2,m_3)=1$, $(m_1,m_2,m_3) \neq (n_1,n_2,n_3)$, $ABC \neq 0$, $EFG \neq 0$ and $n_1,m_1 \geq 9$. Then the equation
 \begin{equation*}
  f(x) = g(y)
 \end{equation*}
has only finitely many integer solutions.
\end{thm*}
\begin{proof}
 If the equation $f(x) = g(y)$ has infinitely many integer solutions, then 
 \begin{equation*}
  f = \varphi \circ f_1 \circ \lambda, \quad g = \varphi \circ g_1 \circ \mu, 
 \end{equation*}
where $\varphi, \mu, \lambda, f_1,g_1 \in \mathbb{Q}[x]$ and $(f_1,g_1)$ is a standard pair, $\mu, \lambda$ are linear polynomials.

Let us consider pairs of the first kind. From symmetry we can assume that
\begin{equation*}
 f_1(x) = x^m, \quad g_1(x) = ax^rp(x)^m  
\end{equation*}
for $a \in \mathbb{Q} \setminus \{0\}$, $0 \leq r < m$, $\gcd(r,m)=1$, $p(x) \in \mathbb{Q}[x]$ and $r+\deg p > 0$. From Theorem \ref{DECOMP} we get
that $\deg \phi \leq 2$ or $\deg f_1 = \deg g_1 =1$. If $\deg \phi = 1$ then we have
\begin{equation*}
\varphi^{-1} \circ f(x) =  x^m \circ \lambda = (ux+v)^m
\end{equation*}
for some $u,v \in \mathbb{Q}$. The polynomial $\varphi^{-1} \circ f$ has at least three non-zero terms so $v \neq 0$. Therefore $\varphi^{-1} \circ f$
has exactly four non-zero terms and $(ux+v)^m$ has exactly $m+1$ non-zero terms. Therefore $m= \deg f = 3$, this is a contradiction with $\deg f \geq 9$.

If $\deg \varphi = 2$. Then from Theorem \ref{DECOMP} we get that $x^m \circ \lambda$ has two or three non-zero terms. As in the previous case
we get that $m=1$ or $m=2$ therefore $\deg f \leq 4$ this contradicts the assumption $\deg f \geq 9$. 

If $\deg f_1 = \deg g_1 = 1$ then $m=1$, $r=0$, $\deg p = 1$. So $f = g \circ l$ for some linear $l$, from Lemma \ref{DZIURY} 
we get that either $(m_1,m_2,m_3) = (n_1,n_2,n_3)$ or $\deg f \leq 4+4-2=6$ - a contradiction. 

Let us consider pairs of the second kind. From symmetry we can assume that $f_1(x) = x^2$, but then from Theorem \ref{DECOMP} we get that $\deg (\varphi) \leq 2$
and therefore $\deg f \leq 4$ - a contradiction with $\deg f \geq 9$. 

Let us consider pairs of the fifth kind. From symmetry we can assume that $f_1(x) = 3x^4-4x^3, g_1 = (ax^2-1)^3$ from Theorem \ref{DECOMP}
we get that $\deg \phi \leq 2$. If $\deg \varphi = 1$ then $\deg f = 4$ contradiction with $\deg f \geq 9$. If $\deg \varphi = 2$ then from Theorem 
\ref{DECOMP} we get that $(ax^2-1)^3 \circ \mu $ has two or three non-zero terms. Let us write $\mu(x) = ux+v$ then we have
$(ax^2-1)^3 \circ \mu = (au^2x^2+2auvx+av^2-1)^3$ has non-zero triple root (note that $au^2x^2+2auvx+av^2-1=au^2x^2$ implies $av^2=1$ and $auv = 0$
contradiction because $u \neq 0$), therefore from Lemma \ref{HAJ} we get a contradiction. 

Let us consider pairs of the third and fourth kind. In both cases we have $f_1 = aD_m(x, \gamma)$ for some 
$a, \gamma \in \mathbb{Q} \setminus \{0\}$. From Theorem \ref{DECOMP} we get that
$ \deg \varphi \leq 2$. If $\deg \varphi = 1$ then we have $D_m(x, \gamma) = \varphi^{-1} \circ f \circ \lambda^{-1}$ and from Lemma \ref{DICKSON} 
we get that $\deg f = m \leq 6$ a contradiction. 
If $ \deg \varphi = 2$ then $f_1 \circ \lambda$ has two or three non-zero terms. Therefore from Lemma  \ref{DICKSON} we get that
$\deg f_1 \leq 4$, so $\deg f \leq 8$ - a contradiction. 

\end{proof}

To prove theorem concerning equations with lacunary polynomials we will use the following result of Zannier \cite{ZAN}.

\begin{thm}[Zannier]
 Suppose that $g, h \in \mathbb{C}[x]$ are non-constant polynomials, that $h(x)$ is not of the
shape $ax^n + b$ and that $g(h(x))$ has at most $l>0$ non-zero terms at positive powers. Then $\deg g \leq 2l(l - 1)$.
\end{thm}

Now we are ready to prove

\begin{thm*}[B]
 Let $l \geq 4$ and $n_1>n_2> \ldots > n_l>0$, $m_1>m_2>m_3>0$ be positive integers. Let
 \begin{equation*}
 f(x) = A_1x^{n_1}+A_2x^{n_2}+\ldots+A_l x^{n_l} + A_{l+1} \quad \text{ and } \quad  g(x) = Ex^{m_1}+Fx^{m_2}+Gx^{m_3}
 \end{equation*}
 be polynomials with rational coefficients such that $\gcd(n_1,n_2,\ldots,n_l) = 1$, $\gcd(m_1,m_2,m_3)=1$,  
 $A_1A_2\ldots A_l \neq 0$, $EFG \neq 0$ and $m_1 \geq 2l(l-1), n_1 \geq 4$. Then the equation 
 \begin{equation*}
  f(x) = g(y)
 \end{equation*}
has only finitely many integer solutions.
 \end{thm*}
\begin{proof}
  If the equation $f(x) = g(y)$ has infinitely many integer solutions, then 
 \begin{equation*}
  f = \varphi \circ f_1 \circ \lambda, \quad g = \varphi \circ g_1 \circ \mu, 
 \end{equation*}
where $\varphi, \mu, \lambda, f_1,g_1 \in \mathbb{Q}[x]$ and $(f_1,g_1)$ is a standard pair, $\mu, \lambda$ are linear polynomials.

Suppose that $\deg \varphi > 2$. From Theorem \ref{DECOMP} we get that either $g_1 \circ \mu = \rho \circ x^d$ for some
linear $\rho \in \mathbb{Q}[x]$, and positive integer $d \geq 2$, or $ \deg(g_1 \circ \mu) = 1$. The first case is impossible
due to the assumption $\gcd(m_1,m_2,m_3)=1$. In the second case we get
\begin{equation*}
 f = \varphi \circ f_1 \circ \lambda = g \circ (g_1 \circ \mu)^{-1} \circ f_1 \circ \lambda.  
\end{equation*}

From Zannier's theorem we get that $ (g_1 \circ \mu)^{-1} \circ f_1 \circ \lambda = ux^k + v$ for some rationals
$u,v$ and positive integer $k$. We claim that $k=1$. Suppose that this is not the case, then all non-zero coefficients of $f$ are at powers divisible by $k$ which contradicts the assumption $\gcd(n_1,\ldots,n_l)=1$. In the case of $k=1$ we have
the equation $f = g \circ \rho$ for some linear $\rho \in \mathbb{Q}[x]$. From Lemma \ref{DZIURY} we get that
\begin{equation*}
2l(l-1) \leq \deg g \leq (l+1)+3 -2 = l+2,
\end{equation*}
or $f(x) = g(ux)$ for some non-zero rational $u$. The former is impossible since $l \geq 4$. The latter is impossible
since $g$ has three non-zero terms and $f$ has at least four non-zero terms.We proved that $\deg \varphi \leq 2$.

Let us consider standard pairs of the first kind. 
Suppose that $g_1(x) = x^m$ for some $m$. If $\deg \varphi = 1$ then we have 
\begin{equation*}
 \varphi^{-1} \circ g = x^m \circ \mu = (ux+v)^m
\end{equation*}
for some rationals $u,v$. Suppose that $v = 0$, then $g(x) = \varphi((ux)^m)$ has at most two non-zero terms - a contradiction. In the case of $v \neq 0$ we get that $(ux+v)^m = \varphi^{-1} \circ g$ has $m+1 = \deg g +1 > 4$ non-zero terms whereas $g$ has three non-zero terms - a contradiction.

If $\deg \varphi = 2$ then from Theorem \ref{DECOMP} we get that $g_1 \circ \mu = x^m \circ \mu$ has two or three non-zero terms, as in previous case we get that $m=1$ or $m=2$. This is a contradiction with the condition  $2m = \deg g > 4$.

Suppose that $(f_1,g_1)$ is a switched pair of the first kind namely $g_1(x) = ax^rp(x)^m, f_1(x)=x^m$.
If $\deg \varphi = 1$ then we have 
\begin{equation*}
 a(ux+v)^r p(ux+v)^m = \varphi^{-1} \circ g
\end{equation*}
for some rationals $u,v$. If $\deg p>0$ then from Lemma \ref{HAJ} we get that $m \leq 3$ which contradicts the fact
that $4<\deg f = \deg (\varphi \circ x^m \circ \lambda) = m$. In the case of $\deg p = 0$ we get that 
$a(ux+v)^r = \varphi^{-1} \circ g$ and again from Lemma \ref{HAJ} we get $r \leq 3$ and thus  $\deg g \leq 3$ - a contradiction.

In the case of $\deg \varphi = 2$ we have 
$\varphi \circ g_1 \circ \mu = g$. We apply Theorem \ref{DECOMP} to get that 
$g_1 \circ \mu$ has at most three non-zero terms. We can write 
\begin{equation*}
g_1 \circ \mu = a(ux+v)^rp(ux+v)^m. 
\end{equation*}

If $\deg p > 0$ then from Lemma \ref{HAJ} we get that $m \leq 2$. However in this case $\deg f = \deg (\varphi \circ x^m \circ \lambda)\leq 4$ - a contradiction. In the case of $\deg p = 0$ we again apply Lemma \ref{HAJ} to get $r \leq 2$. As a consequence 
$\deg g \leq 4$ - a contradiction.  

Observe that $(f_1,g_1)$ cannot be a standard pair of the second kind, since $\deg(\varphi \circ x^2) \leq 4$ and $n_1,m_1>4$.

Let us consider pairs of the third and fourth kind. In both cases we have $g_1 = aD_m(x, \gamma)$ for some 
$a, \gamma \in \mathbb{Q} \setminus \{0\}$.  
If $\deg \varphi = 1$ then we have $aD_m(x, \gamma) = \varphi^{-1} \circ g \circ \lambda^{-1}$ and from Lemma \ref{DICKSON} 
we get that $\deg g = m \leq 6$ a contradiction. 
If $ \deg \varphi = 2$ then $g_1 \circ \lambda$ has two or three non-zero terms. Therefore from Lemma  \ref{DICKSON} we get that
$\deg g_1 \leq 4$ so $\deg g \leq 8$ - a contradiction. 
  
Finally let us note that $(f_1,g_1)$ cannot be a standard pair of the fifth kind. Since $\deg g_1 \leq 6$ and $\deg \varphi \leq 2$
we get that $24 \leq 2l(l-1) \leq \deg g = \deg (\varphi \circ g_1 \circ \mu) \leq 12$ - a contradiction.

\end{proof}

\bigskip

\noindent  Maciej Gawron, Jagiellonian University, Faculty of Mathematics and Computer Science, 
Institute of Mathematics, {\L}ojasiewicza 6, 30 - 348 Krak\'{o}w, Poland\\
e-mail:\;{\tt maciej.gawron@uj.edu.pl}

 \end{document}